\documentclass{article}
\usepackage{amsmath}
\usepackage{amssymb}
\usepackage{amsthm}
\usepackage[all]{xy}
\usepackage{graphicx}
\usepackage{makeidx}
\usepackage{url}
\usepackage{color}

\theoremstyle{theorem}
\newtheorem{theorem}{Theorem}
\newtheorem{lemma}[theorem]{Lemma}
\newtheorem{corollary}[theorem]{Corollary}
\newtheorem{proposition}[theorem]{Proposition}

\theoremstyle{definition}
\newtheorem{definition}{Definition}
\newtheorem{example}{Example}

\theoremstyle{remark}

\def\mA{\mathfrak{A}}
\def\mB{\mathfrak{B}}

\def\EFD{\mbox{\rm EFD}}
\def\EFB{\mbox{\rm EFB}}
\def\+{+^*}

\def\A{\mathcal{A}}

\def\P{\mathcal{P}}
\def\B{\mathcal{B}}

\def\S{\sum^{\#}_{i\in\omega}}
\def\LL{L_{\omega_1\omega}}
\def\EF{Ehrenfeucht-Fra\"{i}ss\'{e} }
\newcommand{\open}{\Bbb}

\newcommand{\oN}{{\open N}}

\renewcommand{\bf}{\textbf}
\renewcommand{\it}{\textit}

\makeindex

\begin{document}
\title{An Ehrenfeucht-Fra\"{i}ss\'{e} Game for $L_{\omega_1\omega}$}

\author{
Jouko V\"a\"an\"anen\thanks{Research partially supported by grant 40734 of the Academy of Finland.}\\ Department of Mathematics and Statistics\\
 University of Helsinki\\
 and\\
  Institute for Logic, Language and Computation,\\ University of Amsterdam\\
  and\\
  Tong Wang\thanks{Research supported by grant 4048 of the European Science Foundation project `New Frontiers of Infinity:  Mathematical, Philosophical and Computational Prospects'.}\\
  Institute for Logic, Language and Computation,\\ University of Amsterdam}

\maketitle

\bibliographystyle{plain}

\begin{abstract}
\EF games are very useful in studying separation and equivalence results in logic. The standard finite \EF game characterizes equivalence in first order logic. The standard \EF game in infinitary logic characterizes equivalence in $L_{\infty\omega}$. The logic $\LL$ is the extension of first order logic with countable conjunctions and disjunctions. There was no \EF game for $\LL$ in the literature.

In this paper we develop an \EF Game for $\LL$. This game is based on a game for propositional and first order logic introduced by Hella and V\"{a}\"{a}n\"{a}nen. Unlike the standard \EF games which are modeled solely after the behavior of quantifiers, this new game also takes into account the behavior of connectives in logic. We prove the adequacy theorem for this game. We also apply the new game to prove complexity results about infinite binary strings.
\end{abstract}

\section{Introduction}
One major limitation of the expressive power of first order logic is that many familiar mathematical concepts (e.g.~finiteness, connectedness of graphs, and well-foundedness) are not definable. These concepts involve, explicitly or implicitly, the notion of infinity which is difficult to handle in first order logic, partly  because first order logic only has \it{finitary} formulas. One way to transcend this boundary is to consider stronger logics where \it{infinitary} formulas are also admissible. Infinitary logics made their first appearance in print with the papers of Scott and Tarski \cite{ScottTarski58} and Tarski \cite{Tarski58}. Among all infinitary logics, the logic $\LL$ turns out to be an especially interesting example. The logic $\LL$ adds to first order logic countable\footnote{In this paper we use `countable' to mean `countably infinite'.} conjunctions and disjunctions of formulas. It is in a sense the smallest infinitary logic. It strikes a nice balance---$\LL$ is expressive enough to define many familiar mathematical concepts which are not definable in first order logic (including the three aforementioned), while it still enjoys model-theoretic properties such as having the Completeness Theorem \cite{Karp64} and the Interpolation Theorem \cite{Lopez-Escobar65}. 

Throughout this paper we assume $L$ to be a relational vocabulary. The terms of $L$ are defined as usual. 
\begin{definition} The set of $\LL$-formulas of $L$ is the smallest set closed under the following operations:
\begin{enumerate}
\item If $t_1$ and $t_2$ are terms, then $t_1=t_2$ is a formula.
\item If $R$ is an $n$-place relational symbol and $t_1,\ldots, t_n$ are terms, then $Rt_1\ldots t_n$ is a formula.
\item If $\phi$ is a formula, then so is $\neg\phi$.
\item If $\phi$ and $\psi$ are formulas, then so is $\phi\vee\psi$.
\item If $\phi$ and $\psi$ are formulas, then so is $\phi\wedge\psi$.
\item If $I$ is a countable set and for every $i\in I$, $\phi_i$ is a formula, then so is $\bigvee_{i\in I}\phi_i$.
\item If $I$ is a countable set and for every $i\in I$, $\phi_i$ is a formula, then so is $\bigwedge_{i\in I}\phi_i$.
\item If $\phi$ is a formula and $x_n$ a variable, then $\exists x_n \phi$ is a formula.
\item If $\phi$ is a formula and $x_n$ a variable, then $\forall x_n \phi$ is a formula.
\end{enumerate}
\end{definition}

Rather surprisingly, some questions in the basic model theory of $\LL$ have been left unanswered. In this paper we address one such question, namely the question of separation in $\LL$. Let us fix some notation. We use $A$, $B$ to denote $L$-structures and use $\A$, $\B$ to denote classes of $L$-structures. We say that a formula $\phi$ \it{separates} a pair of structures $A$ and $B$ if $A\models \phi$ and $B\not\models \phi$. We say that a formula $\phi$ \it{separates} two classes of structures $\mA$ and $\mB$, denoted by $(\mA, \mB)\models \phi$, if for all structures $A\in\mA$ we have $A\models \phi$ and for all structures $B\in \mB$ we have $B\not\models \phi$. Let $A\equiv B$ denote that $A$ and $B$ are elementarily equivalent. If $\alpha$ is an ordinal, let $A\equiv^\alpha_{\infty\omega} B$ denote that $A$ and $B$ satisfy the same $L_{\infty\omega}$ \footnote{The logic $L_{\infty\omega}$ is the extension of first order logic with arbitrary conjunctions and disjunctions. } sentences up to quantifier rank $\alpha$. Finally we denote by $A\equiv_{\infty\omega} B$ that $A$ and $B$ are equivalent in $L_{\infty\omega}$. Note that the notions of separation and equivalence are opposite sides of the same coin: two structures can be separated by a formula in a certain language if and only if they are not equivalent in this language. The central question that we try to answer in this paper is the following:
\begin{description}
\item[Question 1:] Given two structures $A$ and $B$, is there a formula $\phi$ in $\LL$ that separates them? More generally, given two classes of structures $\mA$ and $\mB$, is there a formula $\phi$ in $\LL$ that separates them? 
\end{description}

A very useful machinery in logic to study equivalence and separation results is the \EF game, also known as the back-and-forth game. The \EF game was first formulated for first order logic by Ehrenfeucht \cite{Ehrenfeucht57, Ehrenfeucht6061}, based on the ideas of Fra\"{i}ss\'{e} \cite{Fraisse55}. Let us denote this game by EF$_n$. For $n$ a natural number, the game EF$_n$ characterizes equivalence in first order logic up to formulas of quantifier rank $n$. Karp \cite{Karp64} developed the idea into a characterization of equivalence in infinitary logic using back-and-forth sequences, which was later formulated as a game in \cite{Benda69}. See also \cite[Chapter 7]{Vaananen11}. For $\alpha$ an ordinal, the game EFD$_\alpha$ characterizes equivalence in $L_{\infty\omega}$ up to formulas of quantifier rank $\alpha$. The logic $\LL$ is intermediate between first order logic and $L_{\infty\omega}$. It is only natural to ask: is there an \EF game that, standing somehow in between the standard games EF$_n$ and EFD$_\alpha$, characterizes equivalence in $\LL$? The question of finding such a game has been open since 1970s. 

In this paper we introduce an \EF game for $\LL$. The text is divided as follows. In Section 2 we analyze the reason why the standard game EFD$_\alpha$ fails to characterize equivalence in $\LL$. The idea is roughly that the standard game only models the behavior of quantifiers in logic, and while this is enough for studying separation in $L_{\infty\omega}$, it is not enough for $\LL$ since the latter is defined by explicit reference to the \it{countability} of conjunctions and disjunctions. An \EF game for the latter needs to model also the behavior of connectives. Fortunately there already exists an \EF game for first order logic that is in this spirit, namely the game EFB$_n$ introduced by Hella and V\"{a}\"{a}n\"{a}nen \cite{HellaVaananen}. We present this game in Section 3. In Section 4 we extend this game to cover $\LL$-formulas and we prove the adequacy theorem for the new game. In Section 5 we consider other alternatives for an \EF game for $\LL$. In Section 6 we test our new tool on the complexity of information coded by infinite binary strings.
 
\section{The Standard \EF Game}
Let us first recapitulate the standard \EF game in infinitary logic.

\begin{definition}
Let $L$ be a relational vocabulary and $A$, $B$ be $L$-structures. A partial mapping $\pi: A\rightarrow B$ is a \it{partial isomorphism}\index{partial isomorphism} if it is an isomorphism between $A\upharpoonright{\rm{dom}(\pi)}$ and $B\upharpoonright{\rm{ran}(\pi)}$.
\end{definition}

\begin{definition}
Let $L$ be a relational vocabulary and $A$, $B$ be $L$-structures. Let $\alpha$ be an ordinal. The game EFD$_\alpha(A, B)$ has two players. The number $\alpha$ is called the \it{rank} of the game. The positions in the game are of the form $(\bar{a}, \bar{b}, \gamma)$, where $\bar{a}$ and $\bar{b}$ are tuples in $A$ and $B$ respectively and $\gamma$ is an ordinal. The game begins from the position $(\emptyset, \emptyset, \alpha)$. Suppose the game is in position $(\bar{a}, \bar{b}, \gamma)$, $\gamma>0$, there are the following possibilities for the next move in the game.
\begin{enumerate}
\item Player \bf{I} picks an element $c\in A$. He also picks an ordinal $\beta<\gamma$. Player \bf{II} responds by picking an element $d\in B$. The game continues from the position $(\bar{a}c, \bar{b}d, \beta)$.
\item Player \bf{I} picks an element $d\in B$. He also picks an ordinal $\beta<\gamma$. Player \bf{II} responds by picking an element $c\in A$. The game continues from the position $(\bar{a}c, \bar{b}d, \beta)$.
\end{enumerate}
The game terminates in a position $(\bar{a}, \bar{b}, \gamma)$ if $\gamma=0$. Player \bf{II} wins if $\bar{a}\mapsto \bar{b}$ is a partial isomorphism from $A$ to $B$. Otherwise player \bf{I} wins.
\end{definition}

The following result is essentially due to Karp \cite{Karp64}, although it was originally formulated in terms of back-and-forth sequences.

\begin{theorem}[Karp]
Let $L$ be a relational vocabulary, $A$ and $B$ be $L$-structures and $\alpha$ be an ordinal. Then the following are equivalent:
\begin{enumerate}
\item $A\equiv^\alpha_{\infty\omega} B$.
\item Player \bf{II} has a winning strategy in \EFD$_\alpha(A, B)$.
\end{enumerate}
\end{theorem}

The game EFD$_\alpha$ provides important information about our Question 1. Make the following simple observation.

\begin{proposition}\label{p2}
For any formula $\phi\in \LL$, the quantifier rank of $\phi$ is a countable ordinal.
\end{proposition}

\begin{corollary}\label{c3}
If player \bf{II} has a winning strategy in $\mathrm{EFD}_{\omega_1}(A, B)$, then there is no formula in $\LL$ that separates $A$ and $B$.
\end{corollary}

The standard \EF games EF$_n$ and EFD$_\alpha$ characterize a hierarchy of equivalence relations. They provide in a sense upper bounds and lower bounds for Question 1. The situation can be illustrated by Figure \ref{f1}.

\begin{figure}[h]
  \centering
  \includegraphics[width=6cm, height=5cm]{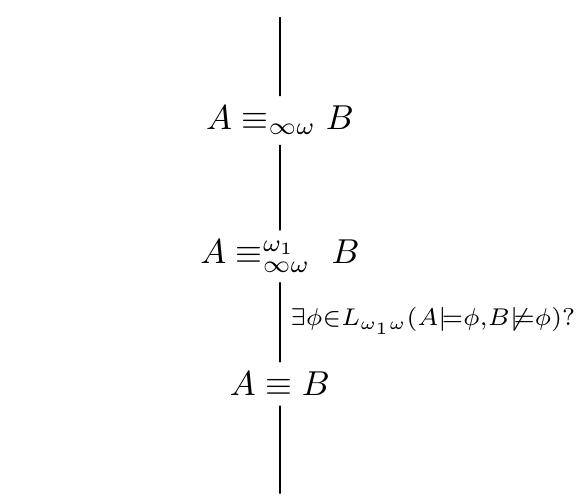}
  \caption{Hierarchy of equivalence relations}
  \label{f1}
\end{figure}

If two structures $A$ and $B$ are not even elementarily equivalent, then obviously there is an $\LL$-formula separating the two because this formulas is already in first order logic. On the other hand, if $A$ and $B$ are $L_{\infty\omega}$-equivalent, then there is no hope of finding an $\LL$-formula separating the two structures. Corollary \ref{c3} shows that we can improve the upper bound: if two structures are equivalent in $L_{\infty\omega}$ up to formulas of quantifier rank $\omega_1$, then they cannot be separated by any formula in $\LL$.

Therefore, from the point of view of $\LL$ all the excitement concentrates in the interval between $A\equiv B$ and $A\equiv^{\omega_1}_{\infty\omega}B$. However the game EFD$_\alpha$ is not able to tell us much about what happens in this interval. The reason is that although a small number of connectives implies a small quantifier rank (as in Proposition \ref{p2}), there is no implication in the reverse direction. We may very well have a separation formula of very low quantifier rank, yet it contains a large number of conjunctions and disjunctions, and the standard game cannot detect the existence of such a formula. In particular, the standard game does not answer the question whether there is a quantifier-free formula in $\LL$ separating two structures, and still, the elimination of quantifiers in favor of connectives is the oldest application of infinitary logic.\footnote{According to Barwise \cite{Barwise81}, Charles Peirce thought of quantifiers as infinite conjunctions or disjunctions, and this was picked up  by L\"owenheim, Wittgenstein and others, and used in proof theory by  Novikoff already in the 1940s.}

The failure of the standard \EF game in characterizing equivalence in $\LL$ can be understood in the following way. An \EF game, meaning a set of rules and winning conditions, corresponds to a way of measuring the complexity of formulas. If a game on the one hand, and a complexity measure on the other hand, form a good match, we would have the adequacy theorem for this game. In the case of the game EFD$_\alpha$, the complexity of a formula is measured by its quantifier rank. The two types of moves in the game EFD$_\alpha$ model the nature of existential and universal quantifiers respectively. The logic $\LL$, however, is defined by reference to the countability of conjunctions and disjunctions, and quantifier rank simply does not capture this information. Therefore there is a sense in saying that the game EFD$_\alpha$ does not solve the question of separation in $\LL$ because it measures the complexity of $\LL$-formulas in an unsuitable way. 

In this light, we need two things in order to develop an \EF game for $\LL$: (a) a new game and (b) a new measure  of the complexity of infinitary formulas. These two components have to be compatible so that we can have the adequacy theorem. Moreover, both of them should take into account not only the behavior of quantifiers, but also of connectives. 

\section{A Game for First Order Logic}
What is a suitable complexity measure for $\LL$-formulas then? We draw inspiration from the game EFB$_n$, proposed by Hella and V\"{a}\"{a}n\"{a}nen in \cite{HellaVaananen} as a refinement of the standard \EF game for first order logic. This game resembles the \EF game for independence friendly logic, presented in \cite{Vaananen02}. An innovative feature of the game EFB$_n$ is that it measures the complexity of a formula by making use of our intuition about the length of a formula. 

Let $L$ be a relational vocabulary as usual. Henceforth we shall assume all formulas to be in negation normal form. 
\begin{definition}\label{d5}
The \it{size}\index{size}, denoted by $\mathrm{s}(\phi)$, of a formula $\phi$ in first order logic is defined inductively as follows:
\begin{displaymath}
\begin{aligned}
\mathrm{s}(\phi)&=1 \textrm{ if } \phi \textrm{ is an atomic or negated atomic formula},\\
\mathrm{s}(\phi \wedge \psi)&= \mathrm{s}(\phi)+\mathrm{s}(\psi),\\
\mathrm{s}(\phi\vee \psi)&=\mathrm{s}(\phi)+\mathrm{s}(\psi),\\
\mathrm{s}(\exists x \phi)&=\mathrm{s}(\phi)+1,\\
\mathrm{s}(\forall x \phi)&= \mathrm{s}(\phi)+1.
\end{aligned}
\end{displaymath}
\end{definition}
The size of a first order formula is a positive integer. It is always bigger or equal to the quantifier rank.

To define the game EFB$_n$ we need some notation. We use $x_j, j\in \oN$ to denote variables. A variable assignment for a structure $A$ is a finite partial mapping $\alpha: \oN\rightarrow A$. The finite domain of $\alpha$ is denoted by dom$(\alpha)$. The game EFB$_n$ is defined on pairs of classes of structures rather than on pairs of structures. We consider classes $\mA$ of structures $(A, \alpha)$, where $A$ is a model and $\alpha$ is a variable assignment on $A$. We assume that whenever $(A, \alpha), (B, \beta)\in \mA$, then $A$ and $B$ have the same vocabulary, and $\alpha$ and $\beta$ have the same domain, which we denote by  dom$(\mA)$. If $\alpha$ is an assignment on $A$, $a\in A$ and $j\in \oN$, then $\alpha(a/j)$ is the assignment that maps $j$ to $a$ and agrees with $\alpha$ otherwise. If $F$ is a choice function on $\mA$, namely that $F$ is a function defined on $\mA$ such that $F(A, \alpha)\in A$ for all $(A, \alpha)\in \mA$, then $\mA(F/j)$ is defined as $\{(A, \alpha(F(A, \alpha)/j)|(A, \alpha)\in \mA\}$. Finally, $\mA(\star/j)=\{(A, \alpha(a/j))|(A, \alpha)\in \mA, a\in A)\}$.

\begin{definition}
Let $L$ be a relational vocabulary and $\mA_0$, $\mB_0$ classes of $L$-structures. Let $n$ be a positive integer. The game \EFB$_n(\mA_0, \mB_0)$\index{EFB$_n$} has two players. The number $n$ is called the \it{rank}\index{rank} of the game. The positions in the game are of the form $(\mA, \mB, m)$ where $\mA, \mB$ are classes of $L$-structures and $m\in \oN$. The game begins from position $(\mA_0, \mB_0, n)$. Suppose the game is in position $(\mA, \mB, m)$. There are the following possibilities for the next move in the game.

\begin{description}
\item[Left splitting move:]Player \bf{I} first represents $\mA$ as a union $\mA_1\cup \mA_2$. He also chooses positive numbers $m_1$ and $m_2$ such that $m_1+m_2=m$. Now the game continues from the position $(\mA_1, \mB, m_1)$ or from the position $(\mA_2, \mB, m_2)$, and player \bf{II} can choose which.

\item[Right splitting move:]Player \bf{II} first represents $\mB$ as a union $\mB_1\cup \mB_2$. He also chooses positive numbers $m_1$ and $m_2$ such that $m_1+m_2=m$. Now the game continues from the position $(\mA, \mB_1, m_1)$ or from the position $(\mA, \mB_2, m_2)$, and player \bf{II} can choose which.

\item[Left supplementing move:] Player \bf{I} picks an element from each structure $(A, \alpha)\in \mA$. More precisely, \bf{I} chooses a natural number $j$ and a choice function $F$ for $\mA$. Then the game continues from the position\newline $(\mA(F/j), \mB(\star/j), m-1)$.

\item[Right supplementing move:]Player \bf{I} picks an element from each structure $(B, \beta)\in \mB$. More precisely, \bf{I} chooses a natural number $j$ and a choice function $F$ for $\mB$. Then the game continues from the position\newline $(\mA(\star/j), \mB(F/j), m-1)$.
\end{description}
The game ends in a position $(\mA, \mB, m)$ and player \bf{I} wins if there is an atomic or a negated atomic formula $\phi$ such that $(\mA, \mB)\models \phi$. Player \bf{II} wins the game if they reach a position $(\mA, \mB, m)$ such that $m=1$ and \bf{I} does not win in this position.
\end{definition}

The game EFB$_n$ characterizes separation in first order logic up to formulas of size $n$:

\begin{theorem}[Hella-V\"{a}\"{a}n\"{a}nen]\label{10}
Let $L$ be a relational vocabulary, $\mA$ and $\mB$ be classes of $L$-structures, and let $n$ be a positive integer. Then the following are equivalent:
\begin{enumerate}
\item Player \bf{I} has a winning strategy in the game \EFB$_n(\mA, \mB)$.
\item There is a first order $L$-formula of size $\leq n$ such that $(\mA, \mB)\models \phi$.
\end{enumerate}
\end{theorem}

The proof of this theorem is postponed to the next section, where we prove the adequacy theorem for the generalized game. For now the interested reader is referred to \cite{HellaVaananen}. 

%The four types of moves in the game EFB$_n$ model the four formula-forming operations in first order logic respectively: the left splitting move for disjunction, the right splitting move for conjunction, the left supplementing move for existential quantification and the right supplementing move for universal quantification. This feature enables us to consider various fragments of the game EFB$_n$ which characterize separation in different fragments of first order logic. For example, if we only allow splitting moves in the game, what we get is essentially a game for propositional logic. 

%\begin{corollary}
%Let \EFB$^P_n(\mA, \mB)$ be the restriction of the game \EFB$_n(\mA, \mB)$ where player \bf{I} can only make splitting moves, but no supplementing moves. Then the following are equivalent:
%\begin{enumerate}
%\item Player \bf{I} has a winning strategy in the game \EFB$^P_n(\mA, \mB)$.
%\item There is a quantifier-free $L$-formula of size $\leq n$ such that $(\mA, \mB)\models \phi$.
%\end{enumerate}
%\end{corollary}

\section{Extending the Game to $L_{\omega_1\omega}$}
The game EFB$_n$ sets up a good model in our quest. Our next step is to extend this game to $\LL$. Again, there are two things that we need to do: a) modifying the rules of the game, and b) extending the definition of size to cover $\LL$-formulas. We shall start from the latter.

At a first glance, the most straightforward way of extending Definition \ref{d5} to infinitary formulas is to leave the clauses for atomic formulas, quantifiers, finite conjunctions and disjunctions unchanged, and define the size of a countable conjunction or disjunction as the supremum of the size of its proper subformulas. However this approach leads to undesirable consequences. Here we encounter the interesting phenomenon that the arithmetic of infinite ordinals behave differently from the finite case. For one thing, addition of infinite ordinals is not commutative. Consider for example a formula $\phi$ with $\mathrm{s}(\phi)=\omega$ and a formula $\psi$ with $\mathrm{s}(\psi)=1$. We would have that $\mathrm{s}(\phi\wedge\psi)=\omega+1$, which is not equal to $\mathrm{s}(\psi\wedge\phi)=1+\omega=\omega$. This is counterintuitive. Moreover, in this example the formula $\psi\wedge \phi$ has the same size as its proper subformula $\phi$, a fact that is fatal to the induction proof of the adequacy theorem. We need a new way of adding up ordinals. 

The operation that we turn to is the \it{natural sum}. The natural sum is defined using the Cantor normal form of ordinals. 

\begin{definition}
Let $\gamma_1$ and $\gamma_2$ be ordinals. One can represent $\gamma_1$ and $\gamma_2$ uniquely in the form
\[\gamma_1=\omega^{\alpha_1}\cdot k_1+\ldots+\omega^{\alpha_n}\cdot k_n
\]
\[\gamma_2=\omega^{\alpha_1}\cdot j_1+\ldots+\omega^{\alpha_n}\cdot j_n
\]
where $\alpha_1>\ldots>\alpha_n$ is a sequence of ordinals, $k_1,\ldots, k_n$ and $j_1,\ldots, j_n$ are natural numbers satisfying $k_i+j_i>0$ for all $i$. 
Define the \it{natural sum} of $\gamma_1$ and $\gamma_2$, denoted by $\gamma_1\#\gamma_2$, of $\gamma_1$ and $\gamma_2$ as
\[\gamma_1\#\gamma_2=\omega^{\alpha_1}\cdot(k_1+j_1)+\ldots+\omega^{\alpha_n}\cdot(k_n+j_n).\]
\end{definition}

The natural sum is commutative and associative. We also introduce the natural sum for countable sequences of ordinals.

\begin{definition}
Let $\{\gamma_i|i\in \omega\}$ be a sequence of ordinals. Let $S_n$ denote the natural sum of the first $n$ items in the sequence:
\[S_n=\gamma_0\#\ldots\#\gamma_{n-1}.\]
Define the \it{infinite natural sum}\index{natural sum!infinite}, denoted by $\S \gamma_i$, of the sequence $\{\gamma_i|i\in \omega\}$ as
\[\S\gamma_i=\sup_{n\in\omega} S_n.\]
\end{definition}

We are now ready to define the notion of size for $\LL$-formulas. Again assume all formulas to be in negation normal form.

\begin{definition}\label{d8}
The \it{size}\index{size}, denoted by $\mathrm{s}(\phi)$, of a formula $\phi$ in $\LL$ is defined inductively as follows:
\begin{displaymath}
\begin{aligned}
\mathrm{s}(\phi)&=1 \textrm{ if } \phi \textrm{ is an atomic or negated atomic formula},\\
\mathrm{s}(\phi \wedge \psi)&= \mathrm{s}(\phi)\#\mathrm{s}(\psi),\\
\mathrm{s}(\phi\vee \psi)&=\mathrm{s}(\phi)\#\mathrm{s}(\psi),\\
\mathrm{s}(\exists x \phi)&=\mathrm{s}(\phi)+1,\\
\mathrm{s}(\forall x \phi)&= \mathrm{s}(\phi)+1,\\
\mathrm{s}(\bigwedge_{i\in \omega} \phi_i)&=\S \mathrm{s}(\phi_i),\\
\mathrm{s}(\bigvee_{i\in \omega} \phi_i)&=\S \mathrm{s}(\phi_i).
\end{aligned}
\end{displaymath}
\end{definition}

If $I$ is any countable set and for each $i\in I$, $\phi_i$ is a formula, we define $\mathrm{s}(\bigwedge_{i\in I} \phi_i)$ to be $\mathrm{s}(\bigwedge_{i\in\omega}\phi_{p(i)})$, where $p$ is a bijection from $\omega$ to $I$, and similarly for disjunction.  In the next proposition we show that this definition does not depend on the choice of $p$. The size of an $\LL$-formula is always a countable ordinal. The following are straightforward properties of size.

\begin{proposition}\label{p5}
\begin{enumerate}
\item If $\phi$ is a first order formula, then Definition \ref{d5} and Definition \ref{d8} assign it the same size.
\item If $\phi$ is a proper subformula of $\psi$, then $\mathrm{s}(\phi)<\mathrm{s}(\psi)$.
\item The size of a conjunction is invariant under permutations of the conjuncts. More precisely, if $\phi$ and $\psi$ are formulas then $\mathrm{s}(\phi\wedge\psi)=\mathrm{s}(\psi\wedge\phi)$. If $\phi_i, i\in \omega$ is a sequence of formulas and $p$ is a permutation of $\omega$, then $\mathrm{s}(\bigwedge_{i\in\omega}\phi_i)=\mathrm{s}(\bigwedge_{i\in\omega}\phi_{p(i)})$. The same applies to disjunctions. 
\end{enumerate}
\end{proposition}

\begin{proof}
1. When restricted to finite ordinals, the natural sum coincides with the usual ordinal sum.

2. Every formula has a non-zero size. If $\gamma_1$ and $\gamma_2$ are non-zero ordinals, we always have $\gamma_1<\gamma_1\#\gamma_2$ and $\gamma_2<\gamma_1\#\gamma_2$. Hence the size of a conjunction is always greater than both of the conjuncts. The same applies to disjunctions. The quantifier case is obvious.

3. The natural sum of ordinals is commutative. For the infinitary case, it suffices to observe that every finite sum $S_n$ of the sequence $\{\mathrm{s}(\phi_i)| i\in\omega\}$ is subsumed by a finite sum of the sequence $\{\mathrm{s}(\phi_{p(i)})| i\in\omega\}$.
\end{proof}

Now we move on to the other half of our plan, namely adjusting the rules of the game EFB$_n$ for $\LL$. We add, in addition to the four types of moves in EFB$_n$, two new types of moves: the infinite left splitting move and the infinite right splitting move. These two types of moves are supposed to model infinite conjunctions and disjunctions.

\begin{definition}
Let $L$ be a relational vocabulary and $\mA_0$, $\mB_0$ classes of $L$-structures. Let $\alpha$ be a countable ordinal. The game \EFB$_\alpha(\mA_0, \mB_0)$\index{EFB$_\alpha$} has two players. The number $\alpha$ is called the \it{rank}\index{rank} of the game. The positions in the game are of the form $(\mA, \mB, \gamma)$ where $\mA, \mB$ are classes of $L$-structures and $\gamma$ is an ordinal. The game begins from position $(\mA_0, \mB_0, \alpha)$. Suppose the game is in position $(\mA, \mB, \gamma)$. There are the following possibilities for the next move in the game.

\begin{description}

\item[Finite left splitting move:]Player \bf{I} first represents $\mA$ as a union $\mA_1\cup \mA_2$. He also chooses non-zero ordinals $\gamma_1$ and $\gamma_2$ such that $\gamma_1\#\gamma_2\leq \gamma$. Now the game continues from the position $(\mA_1, \mB, \gamma_1)$ or from the position $(\mA_2, \mB, \gamma_2)$, and player \bf{II} can choose which.

\item[Finite right splitting move:]Player \bf{I} first represents $\mB$ as a union $\mB_1\cup \mB_2$. He also chooses non-zero ordinals $\gamma_1$ and $\gamma_2$ such that $\gamma_1\#\gamma_2\leq \gamma$. Now the game continues from the position $(\mA, \mB_1, \gamma_1)$ or from the position $(\mA, \mB_2, \gamma_2)$, and player \bf{II} can choose which.

\item[Infinite left splitting move:]Player \bf{I} first represents $\mA$ as a countable union $\bigcup_{i\in\omega} \mA_i$. He also chooses a sequence of non-zero ordinals $\{\gamma_i|i\in \omega\}$ such that $\S\gamma_i\leq \gamma$. Now the game continues from the position $(\mA_i, \mB, \gamma_i)$ for some $i\in \omega$, and player \bf{II} can choose which.

\item[Infinite right splitting move:]Player \bf{I} first represents $\mB$ as a countable union $\bigcup_{i\in\omega} \mB_i$.  He also chooses a sequence of non-zero ordinals $\{\gamma_i|i\in \omega\}$ such that $\S\gamma_i\leq \gamma$. Now the game continues from the position $(\mA, \mB_i, \gamma_i)$ for some $i\in \omega$, and player \bf{II} can choose which.

\item[Left supplementing move:] Player \bf{I} picks an element from each structure $(A, \alpha)\in \mA$. More precisely, \bf{I} chooses a natural number $j$ and a choice function $F$ for $\mA$. He also chooses an ordinal $\delta<\gamma$. Then the game continues from the position $(\mA(F/j), \mB(\star/j), \delta)$.

\item[Right supplementing move:]Player \bf{I} picks an element from each structure $(B, \beta)\in \mB$. More precisely, \bf{I} chooses a natural number $j$ and a choice function $F$ for $\mB_m$. He also chooses an ordinal $\delta<\gamma$. Then the game continues from the position $(\mA(\star/j), \mB(F/j), \delta)$.
\end{description}
The game ends in a position $(\mA, \mB, \gamma)$ and player \bf{I} wins if there is an atomic or a negated atomic formula $\phi$ such that $(\mA, \mB)\models \phi$. Player \bf{II} wins the game if they reach a position $(\mA, \mB, \gamma)$ such that $\gamma=1$ and \bf{I} does not win in this position.
\end{definition}

The game EFB$_\alpha$ characterizes separation in $\LL$ up to size $\alpha$. The following theorem is the central result in this paper.

\begin{theorem}[Adequacy Theorem for EFB$_\alpha$]\label{t7}
Let $L$ be a relational vocabulary,\footnote{Note that $L$ is not necessarily countable.} $\mA$ and $\mB$ classes of $L$-structures, and let $\alpha$ be a countable ordinal. Then the following are equivalent:
\begin{enumerate}
\item Player \bf{I} has a winning strategy in the game \EFB$_\alpha(\mA, \mB)$.
\item There is an $L$-formula $\phi$ in $\LL$ of size $\leq \alpha$ such that $(\mA, \mB)\models \phi$.
\end{enumerate}
\end{theorem}

\begin{proof}
We prove this theorem by induction on the rank $\alpha$. When $\alpha=1$ the proposition is obvious. Suppose the proposition is true for all ordinals $\alpha<\gamma$. Now consider the case $\alpha=\gamma$. 

(1$\Longrightarrow$ 2) Suppose \bf{I} has a winning strategy in the game EFB$_\gamma(\mA, \mB)$. We look at the first move in this strategy. Depending on which type of move it is, there are the following possibilities.

\bf{Case 1.} Player \bf{I} first makes a finite left splitting move. He represents $\mA$ as a union  $\mA_1\cup \mA_2$. He chooses non-zero ordinals $\gamma_1$ and $\gamma_2$ such that $\gamma_1\#\gamma_2\leq \gamma$. Now since this is a winning strategy, both $(\mA_1, \mB, \gamma_1)$ and $(\mA_2, \mB, \gamma_2)$ are winning positions for \bf{I}. Note that both $\gamma_1$ and $\gamma_2$ are strictly smaller than $\gamma$. By the induction hypothesis, there are $L$-formulas $\psi$ and $\theta$ in $\LL$ such that $\mathrm{s}(\psi)\leq \gamma_1$, $\mathrm{s}(\theta)\leq \gamma_2$, $(\mA_1, \mB)\models \psi$ and $(\mA_2, \mB)\models \theta$. Let $\phi$ be the formula $\psi\vee \theta$. Then we have that $\mathrm{s}(\phi)=\mathrm{s}(\psi)\#\mathrm{s}(\theta)\leq\gamma_1\#\gamma_2\leq \gamma$, and that $(\mA, \mB)\models \phi$.

\bf{Case 2.} Player \bf{I} first makes a finite right splitting move. He represents $\mB$ as a union  $\mB_1\cup \mB_2$. He chooses non-zero ordinals $\gamma_1$ and $\gamma_2$ such that $\gamma_1\#\gamma_2\leq\gamma$. Now since this is a winning strategy, both $(\mA, \mB_1, \gamma_1)$ and $(\mA, \mB_2, \gamma_2)$ are winning positions for \bf{I}. Note that both $\gamma_1$ and $\gamma_2$ are strictly smaller than $\gamma$. By the induction hypothesis, there are formulas $\psi$ and $\theta$ such that $\mathrm{s}(\psi)\leq \gamma_1$, $\mathrm{s}(\theta)\leq \gamma_2$, $(\mA, \mB_1)\models \psi$ and $(\mA, \mB_2)\models \theta$. Let $\phi$ be the formula $\psi\wedge \theta$. Then we have that $\mathrm{s}(\phi)= \mathrm{s}(\psi)\#\mathrm{s}(\theta)\leq \gamma_1\#\gamma_2\leq\gamma$, and that $(\mA, \mB)\models \phi$.

\bf{Case 3.} Player \bf{I} first makes an infinite left splitting move. He represents $\mA$ as a union $\bigcup_{i\in\omega} \mA_i$. He chooses ordinals $\{\gamma_i|i\in \omega\}$ such that $\S\gamma_i\leq \gamma$. Note that every $\gamma_i$ is strictly smaller than $\gamma$. Now since this is a winning strategy, for every $i\in \omega$ $(\mA_i, \mB, \gamma_i)$ is a winning position for him. By the induction hypothesis, there are $L$-formulas $\phi_i$ in $\LL$ such that $\mathrm{s}(\phi_i)\leq \gamma_i$ and $(\mA_i, \mB)\models \phi_i$ for all $i\in \omega$. Let $\phi$ be the formula $\bigvee_{i\in \omega} \phi_i$. It is an $\LL$-formula. We have that $\mathrm{s}(\phi)=\S \mathrm{s}(\phi_i)\leq\S \gamma_i\leq \gamma$, and that $(\mA, \mB)\models \phi$.

\bf{Case 4.} Player \bf{I} first makes an infinite right splitting move. He represents $\mB$ as a union $\bigcup_{i\in\omega} \mB_i$. He chooses ordinals $\{\gamma_i|i\in \omega\}$ such that $\S\gamma_i\leq \gamma$. Note that every $\gamma_i$ is strictly smaller than $\gamma$. Now since this is a winning strategy, for every $i\in \omega$ $(\mA, \mB_i, \gamma_i)$ is a winning position for him. By the induction hypothesis, there are $L$-formulas $\phi_i$ in $\LL$ such that $\mathrm{s}(\phi_i)\leq \gamma_i$ and $(\mA, \mB_i)\models \phi_i$ for all $i\in \omega$. Let $\phi$ be the formula $\bigwedge_{i\in \omega} \phi_i$. It is an $\LL$-formula. We have that $\mathrm{s}(\phi)=\S \mathrm{s}(\phi_i)\leq\S \gamma_i\leq \gamma$, and that $(\mA, \mB)\models \phi$.

\bf{Case 5.} Player \bf{I} first makes a left supplementing move. He chooses a natural number $j$ and a choice function $F$ for $\mA$. He also chooses an ordinal $\delta<\gamma$. The game continues from $(\mA(F/j), \mB(\star/j), \delta)$. Since this is a winning strategy for \bf{I}, the position $(\mA(F/j), \mB(\star/j), \delta)$ is a winning position for him too. By the induction hypothesis, there is a formula $\psi$ such that $\mathrm{s}(\psi)\leq \delta$, and $(\mA(F/j), \mB(\star/j))\models \psi$. Let $\phi$ be the formula $\exists x_j\psi$. Then we have that $\mathrm{s}(\phi)=\mathrm{s}(\psi)+1\leq \delta+1\leq \gamma$, and that $(\mA, \mB)\models \phi$.

\bf{Case 6.} Player \bf{I} first makes a right supplementing move. He chooses a natural number $j$ and a choice function $F$ for $\mB$. He also chooses an ordinal $\delta<\gamma$. The game continues from $(\mA(\star/j), \mB(F/j), \delta)$. Since this is a winning strategy for \bf{I}, the position $(\mA(\star/j), \mB(F/j), \delta)$ is a winning position for him too. By the induction hypothesis, there is a formula $\psi$ such that $\mathrm{s}(\psi)\leq \delta$, and $(\mA(\star/j), \mB(F/j))\models \psi$. Let $\phi$ be the formula $\forall x_j\psi$. Then we have that $\mathrm{s}(\phi)=\mathrm{s}(\psi)+1\leq \delta+1\leq \gamma$, and that $(\mA, \mB)\models \phi$. \newline

Now for the converse direction (2 $\Longrightarrow$ 1). Suppose there is an $\LL$-formula $\phi$ of size $\leq \gamma$ such that $(\mA, \mB)\models \phi$. Depending on the shape of $\phi$, there are the following possibilities. 
 
 \bf{Case 1.} The formulas $\phi$ is an atomic formula. By definition player \bf{I} wins the game EFB$_1(\mA, \mB)$.
 
 \bf{Case 2.} The formula $\phi$ is $\psi\vee\theta$. Let $\mA_1$ be the class of structures $(A, \alpha)\in \mA$ such that $(A, \alpha)\models \psi$, $\mA_2$ be the class of structures $(A, \alpha)\in \mA$ such that 
$(A, \alpha)\models \theta$. Since $(A, \alpha)\models \phi$ for all $(A, \alpha)\in \mA$, we have $\mA=\mA_1\cup \mA_2$. Moreover since  $(B, \beta)\not\models \phi$ for all $(B, \beta)\in \mB$, we have $(\mA_1, \mB)\models \psi$ and $(\mA_2, \mB)\models \theta$. Finally, as $\mathrm{s}(\phi)\leq \gamma$, there are ordinals $\gamma_1$ and $\gamma_2$ such that $\mathrm{s}(\psi)\leq \gamma_1$, $\mathrm{s}(\theta)\leq \gamma_2$, and $\gamma_1\#\gamma_2\leq \gamma$.  By the induction hypothesis, \bf{I} has a winning strategy in both $(\mA_1, \mB, \gamma_1)$ and $(\mA_2, \mB, \gamma_2)$. Therefore he has a winning strategy in $(\mA, \mB, \gamma)$ by first making a finite left splitting move, and then follow the winning strategy in $(\mA_1, \mB, \gamma_1)$ or $(\mA_2, \mB, \gamma_2)$.

\bf{Case 3.} The formula $\phi$ is $\psi\wedge\theta$. This case is dual to case 2 and we omit the details.

\bf{Case 4.} The formula $\phi$ is $\bigvee_{i\in\omega}\phi_i$. For each $i$, let $\mA_i$ be the class of structures $(A, \alpha)\in \mA$ such that $(A, \alpha)\models \phi_i$. Since $(A, \alpha)\models \phi$ for all $(A, \alpha)\in \mA$, we have $\mA=\bigcup_{i\in\omega}\mA_i$. Moreover since  $(B, \beta)\not\models\phi$ for all $(B, \beta)\in \mB$, we have $(\mA_i, \mB)\models \phi_i$ for every $i$. Finally, as $\mathrm{s}(\phi)\leq \gamma$, there are ordinals $\{\gamma_i|i\in \omega\}$ such that $\mathrm{s}(\phi_i)\leq \gamma_i$ and $\S\gamma_i\leq \gamma$. By the induction hypothesis, \bf{I} has a winning strategy in $(\mA_i, \mB, \gamma_i)$ for every $i$. Therefore he has a winning strategy in $(\mA, \mB, \gamma)$ by first making an infinite left splitting move, and then follow the winning strategy in some $(\mA_i, \mB, \gamma_i)$.

\bf{Case 5.} The formula $\phi$ is $\bigwedge_{i\in\omega}\phi_i$. This case is dual to case 4 and we omit the details. 

\bf{Case 6.} The formula $\phi$ is $\exists x_j \psi$. Since $(A, \alpha)\models \phi$ for all $(A, \alpha)\in \mA$, there is a choice function $F$ for $\mA$ such that $(A, \alpha(F(A, \alpha)/j))\models \psi$. Thus $(A, \alpha^*)\models \psi$ for all $(A, \alpha^*)\in \mA(F/j)$. On the other hand, we have that $(\mB, \beta^*)\not\models \psi$ for all $(\mB, \beta^*)\in \mB(\star/j)$. Therefore $(\mA(F/j), \mB(\star/j))\models \psi$. By the induction hypothesis, player \bf{I} has a winning strategy in  $(\mA(F/j), \mB(\star/j), \mathrm{s}(\psi))$. Note that $\mathrm{s}(\phi)=\mathrm{s}(\psi)+1$. Hence \bf{I} has a winning strategy in $(\mA, \mB, \gamma)$ by first making a left supplementing move, chooses the ordinal $\mathrm{s}(\psi)<\mathrm{s}(\phi)\leq\gamma$ and then follow the winning strategy in $(\mA(F/j), \mB(\star/j), \mathrm{s}(\psi))$.

\bf{Case 7.} The formula $\phi$ is $\forall x_j \psi$. This case is dual to case 6 and we omit the details.
\end{proof}

\begin{corollary}\label{c8}
Given two classes of $L$-structures $\mA$ and $\mA'$, let $\mA\subseteq \mA'$ denote that every structure in $\mA$ is contained in $\mA'$. Let $\mA, \mA', \mB, \mB'$ be classes of $L$-structures and $\alpha$ be a countable ordinal. Suppose $\mA'\subseteq\mA$ and $\mB'\subseteq\mB$. If player \bf{I} has a winning strategy in \EFB$_\alpha(\mA, \mB)$, then he also has a winning strategy in \EFB$_\alpha(\mA', \mB')$.
\end{corollary}
\begin{proof}
Suppose \bf{I} has a winning strategy in EFB$_\alpha(\mA, \mB)$. By Theorem \ref{t7}, there is a formula $\phi$ with $\mathrm{s}(\phi)\leq \alpha$ such that $(\mA, \mB)\models \phi$. Since $\mA'\subseteq \mA$ and $\mB'\subseteq \mB$, we also have that $(\mA', \mB')\models \phi$. By Theorem \ref{t7} again, \bf{I} has a winning strategy in EFB$_\alpha(\mA', \mB')$.
\end{proof}

The six types of moves in the game EFB$_\alpha$ model the six formula-forming operations in $\LL$ respectively. The finite left splitting move corresponds to finite disjunction, the finite right splitting move to finite conjunction, the infinite left splitting move to infinite disjunction, the infinite right splitting move to infinite conjunction, the left supplementing move to existential quantification, and the right supplementing move to universal quantification. This feature enables us to consider various fragments of the game EFB$_\alpha$ which characterize separation in different fragments of $\LL$. If we forbid infinite splitting moves, what we get is essentially the old game EFB$_n$. If we only allow splitting moves in the game, the result is a game for the propositional fragment of $\LL$. 

\begin{corollary}\label{c9}
Let \EFB$^P_\alpha(\mA, \mB)$ be the restriction of the game \EFB$_\alpha(\mA, \mB)$ where player \bf{I} can only make finite and infinite splitting moves, but no supplementing moves. Then the following are equivalent.
\begin{enumerate}
\item Player \bf{I} has a winning strategy in the game \EFB$^P_\alpha(\mA, \mB)$.
\item There is a quantifier-free $\LL$-formula $\phi$ of size $\leq \alpha$ such that $(\mA, \mB)\models \phi$.
\end{enumerate}
\end{corollary}

\section{Alternative Conceptions of Size}
Let us take some time to reflect on our result. We have shown that measuring the complexity of formulas by size is \it{a} way to define an \EF game for $\LL$. However, our viewpoint naturally suggests the possibility of a plurality of complexity measures, and hence a plurality of \EF games for $\LL$. Can we find other complexity measures which give rise to an \EF game for $\LL$? Furthermore, to which extent are these complexity measures \it{reasonable}?

We shall start by looking at the bare minimum, namely which complexity measures could give rise to an \EF game. It turns out that if we merely want an adequacy theorem to hold, we can be quite relaxed about the choice of complexity measures. 

\begin{definition}
Let $q:\mathrm{Ord}\rightarrow \mathrm{Ord}$, $r:\mathrm{Ord}^2\rightarrow \mathrm{Ord}$ and $t:\mathrm{Ord}^\omega\rightarrow \mathrm{Ord}$ be ordinal functions. Consider an abstract complexity measure $c$ defined by $q$, $r$ and $t$:
\begin{displaymath}\label{e1}
\begin{aligned}
c(\phi)&=1 \textrm{ if } \phi \textrm{ is an atomic or negated atomic formula},\\
c(\phi \wedge \psi)&= r((c(\phi),c(\psi)),\\
c(\phi\vee \psi)&=r((c(\phi),c(\psi)),\\
c(\exists x \phi)&=q(c(\phi)),\\
c(\forall x \phi)&= q(c(\phi)),\\
c(\bigwedge_{i\in \omega} \phi_i)&=t( \{c(\phi_i)|i\in\omega\}),\\
c(\bigvee_{i\in \omega} \phi_i)&=t(\{c(\phi_i)|i\in\omega\}).
\end{aligned}
\end{displaymath}

We say that $c$ is \it{nice} if $q$, $r$ and $t$ satisfy the following conditions:
\begin{enumerate}
 \item If $\gamma>0$, then $q(\gamma)>\gamma$.
\item If $\gamma_1,\gamma_2>0$, then $r(\gamma_1, \gamma_2)>\gamma_1$ and $r(\gamma_1, \gamma_2)>\gamma_2$.
 \item If $\{\gamma_i|i\in\omega\}$ is a sequence of non-zero ordinals, then $t(\{\gamma_i|i\in\omega\})>\gamma_i$ for all $i\in\omega$. 
\end{enumerate}
\end{definition}

Niceness guarantees that the complexity of a formula is always strictly larger than the complexity of its proper subformulas. Granted this, we can generalize the game EFB$_\alpha$ to any nice complexity measure by performing minor cosmetic surgery. Let EFB$^c_\alpha$ denote the \EF game for $\LL$ under the complexity measure $c$. The starting position and the winning condition in the game EFB$^c_\alpha$ remain the same as in EFB$_\alpha$. The way that the rank of the game is counted is modified. We only give the `left' moves in the game EFB$^c_\alpha$ here; the `right' moves are dual.

\begin{description}
\item[Finite left splitting move:]Player \bf{I} first represents $\mA$ as a union $\mA_1\cup \mA_2$. He also chooses non-zero ordinals $\gamma_1$ and $\gamma_2$ such that $r(\gamma_1,\gamma_2)\leq \gamma$. Now the game continues from the position $(\mA_1, \mB, \gamma_1)$ or from the position $(\mA_2, \mB, \gamma_2)$, and player \bf{II} can choose which.

\item[Infinite left splitting move:]Player \bf{I} first represents $\mA$ as a countable union $\bigcup_{i\in\omega} \mA_i$. He also chooses a sequence of non-zero ordinals $\{\gamma_i|i\in \omega\}$ such that $t(\{\gamma_i|i\in\omega\})\leq \gamma$. Now the game continues from the position $(\mA_i, \mB, \gamma_i)$ for some $i\in \omega$, and player \bf{II} can choose which.

\item[Left supplementing move:] Player \bf{I} picks an element from each structure $(A, \alpha)\in \mA$. More precisely, \bf{I} chooses a natural number $j$ and a choice function $F$ for $\mA$. He also chooses an ordinal $\delta$ such that $q(\delta)\leq\gamma$. Then the game continues from the position $(\mA(F/j), \mB(\star/j), \delta)$.
\end{description}

It is an easy observation that the proof of the adequacy theorem goes through if we substitute size for any nice complexity measure. We have the following generalized adequacy theorem.
\begin{theorem}
Let $c$ be a nice complexity measure. Let $L$ be a relational vocabulary, $\mA$ and $\mB$ classes of $L$-structures, and let $\alpha$ be an ordinal. Then the following are equivalent:
\begin{enumerate}
\item Player \bf{I} has a winning strategy in the game \EFB$^c_\alpha(\mA, \mB)$.
\item There is an $L$-formula $\phi$ in $\LL$ with $c(\phi)\leq\alpha$ such that $(\mA, \mB)\models \phi$.
\end{enumerate}
\end{theorem}

In particular, size is a nice complexity measure: take $q$ to be the successor function, $r$ the natural sum, and $t$ the infinite natural sum. Another interesting nice complexity measure $c_1$ can be obtained by taking $q$ to be the successor function, $r(\gamma_1, \gamma_2)=\sup(\gamma_1+1, \gamma_2+1)$ and $t(\{\gamma_i|i\in\omega\})=\sup_{i\in\omega}(\gamma_i+1)$. Note that if we take $q$ to be the successor function and $r, t$ to be the supremum, then the complexity measure is just quantifier rank---but this is not a nice complexity measure. 

To determine which of the nice complexity measures are the more natural and reasonable ones is a subtler question. As a general rule of thumb, we would like a complexity measure to carry some meaning, so that we have an idea of what exactly it is measuring. We also want a complexity measure to be invariant under certain logical equivalences, so that it is compatible with our intuition about `complexity'. Between size and $c_1$, size has the advantage of having a clear intuition behind it. This is especially so in the propositional case: the size of a propositional formula is just the number of propositional symbols in this formula. The quantifier and infinitary cases are natural extensions of this basic intuition. For $c_1$ the intuition is not so clear. 
As for invariance under logical equivalences, both size and $c_1$ are invariant under permutations of conjuncts and disjuncts. They part ways, however, when we try to count the complexity of repetitions of formulas. Consider atomic formulas $\phi_i, i\in\omega$. Let $\Theta$ be the infinite conjunction $\bigwedge_{i\in\omega}\phi_0$. Let $\Psi$ be the infinite conjunction $\bigwedge_{i\in\omega}\phi_i$. Size assigns $\Theta$ and $\Psi$ the complexity $\omega$. The measure $c_1$ assigns $\Theta$ and $\Psi$ the complexity $2$. It seems that $c_1$ gives the more reasonable complexity for $\Theta$, since $\Theta$ is no more than a repetition of the same formula $\phi_0$. On the other hand size seems to give the more reasonable complexity for $\Psi$, while $c_1$ does not reflect the fact that $\Psi$ contains much more information than any finite conjunction. However as long as complexity measures are defined purely syntactically, every complexity measure would assign the same complexity to $\Theta$ and $\Psi$, so there is no complexity measure that assigns the `right' complexity to $\Theta$ and $\Psi$ at the same time. This suggests that there is no point in trying to find \it{the} reasonable complexity measure. We may indeed adopt a pluralistic view: we shall choose which is the appropriate complexity measure depending on what information we intend to model.

\section{An Application: the Complexity of Infinite Binary String Properties}

In this section we apply the newly defined game EFB$_\alpha$ to solve one instance of Question 1. We focus on information code by \textit{infinite binary strings}. An infinitary binary string $t$ is a function from $\omega$ to $2$. We can think of a string $t$ simply as a sequence of $0$'s and $1$'s. A property $\P$ of infinite binary strings is a subset $\P\subset 2^\omega$. 

Infinite binary strings can be studied using the propositional fragment of $\LL$. Let $L$ be the propositional language with countably many propositional symbols $p_i$, $i\in \omega$. For an infinitary binary string $t$ and a propositional symbol $p_i$, we say that 
\[t\models p_i\ \textrm{if}\ t(i)=1,\]
and
\[t\models \neg p_i\ \textrm{if} \ t(i)=0.\]
Intuitively, $p_i$ has the intended meaning that `the i'th position of $t$ is $1$', and $\neg p_i$ has the intended meaning that `the i'th position of $t$ is $0$'. We pose the question:
\begin{description}
\item[Question 2:] Given a string property $\P$, what is the minimal size of a propositional $\LL$-formula defining $\P$?
\end{description} 
Note that an infinite binary string $t$ codes a real number. Hence this is essentially a question about the definability of sets of real numbers.

There is a straightforward way to define a string property $\P$. An infinite binary string $t$ can be characterized by the following sentence in $\LL$:
\[ \theta_t=\bigwedge_{i\in\omega}p_i^t,\]
where 
\begin{displaymath}
p_i^t=\left\{ \begin{array}{ll} 
p_i & \textrm{if $t(i)=1$,}\\
\neg p_i & \textrm{if $t(i)=0$.}\end{array}\right.
\end{displaymath}
A property $\P$ of infinite binary strings can be defined by the following sentence:
\[\Phi_\P=\bigvee_{t\in\P} \theta_t.
\]
In general $\Phi_\P$ is an $L_{\infty\omega}$-sentence. When the cardinality of $\P$ is countable, this sentence is in $\LL$. In this case the sentence $\Phi_\P$ has size
\[\mathrm{s}(\Phi_\P)=\sum_{t\in\P}^\#\mathrm{s}(\theta_t)=\sum_{i\in\omega}^\#\omega=\omega^2.\]
It is easy to see that if the complement of $\P$ in $2^\omega$ is countable, $\P$ can also be defined by an $\LL$-sentence of size $\omega^2$. The question is whether the sentence $\Phi_\P$ is minimal.

Regarding each infinite binary string as a `model', we can play the game \EFB$^P_\alpha$ on classes of infinite binary strings. This game solves Question 2 for an interesting class of string properties.
\begin{definition}
A property $\P$ of infinite binary strings is said to be \it{dense} if it is dense in the space\footnote{Called the Cantor space.} $2^\omega$ with the product topology.
\end{definition}

\begin{theorem}\label{t10}
Let $\P_1$ and $\P_2$ be dense properties of infinite binary strings. Player \bf{II} has a winning strategy in \EFB$_\alpha^P(\P_1, \P_2)$ for all $\alpha<\omega^2$.
\end{theorem}

Together with Corollary \ref{c9}, this leads to the following theorem. 
\begin{theorem}
If $\P_1$ and $\P_2$ are dense properties of infinitary binary strings and $\phi$ is a propositional $\LL$-formula separating $\P_1$ and $\P_2$, then the size of $\phi$ is at least $\omega^2$.
\end{theorem}

\begin{corollary}\label{c11}
If $\P$ is a string property such that both $\P$ and its complement are dense, then $\P$ cannot be defined by a propositional $\LL$-formula of size less than $\omega^2$.
\end{corollary}

If a string property $\P$ is countable, then by Cantor diagonal argument its complement is dense. For a countable dense property $\P$ the standard formula $\Phi_\P$ is a minimal formula defining $\P$.

In the remaining of this section we give a proof of Theorem \ref{t10}. Let us first fix some notation. Let $2^{<\omega}$ denote the set of partial functions $\omega \rightarrow 2$ with a finite domain. We call a function $g\in 2^{<\omega}$ a \it{finite segment}. For an infinite binary string $h\in 2^\omega$ and a finite segment $g\in 2^{<\omega}$, we say that $h$ \it{agrees with} $g$, or that $h$ \it{extends} $g$, if $g\subset h$. A string property $\P$ is dense if and only if for any finite segment $g\in 2^{<\omega}$, there exists $h\in\P$ that extends $g$.

In the game EFB$^P_\alpha$ the players are not allowed to play supplementing moves, but can only play splitting moves. When player \bf{I} plays an infinite splitting move, he splits a class of structures $\mA$ (or $\mB$) into countably many pieces $\mA_i, i\in\omega$ (or $\mB_i, i\in \omega$). We say that an infinite splitting move is a \it{proper infinite splitting move} if there are infinitely many pieces $\mA_i$ (or $\mB_i$) that are non-empty. If player \bf{I} splits $\mA$ (or $\mB$) into infinitely many pieces among which only finitely many are non-empty, we call it a \it{degenerate infinite splitting move}\index{infinite splitting move!degenerate}. Strange as it may seem, there is nothing in the rules that forbids \bf{I} playing in this way. A degenerate infinite splitting move is closer in nature to a finite series of finite splitting moves. Henceforth we refer to degenerate infinite splitting move and finite series of finite splitting moves as \it{generalized finite splitting moves}.

Let $\alpha$ be an ordinal smaller than $\omega^2$. We describe the following strategy for player \bf{II} in the game \EFB$_\alpha^P(\P_1, \P_2)$. \newline

\bf{Player II's strategy}: Let $(\mA, \mB, \gamma)$ denote a game position. In each round, player \bf{II} makes sure that
\begin{equation}\label{e9}
\begin{split}
\exists f\in 2^{<\omega} 
\Bigg(&\forall h\in \mA  (h\supset f) \wedge \\
&\forall h \in \mB  (h\supset f)\wedge \\ 
&\forall g\in 2^{<\omega}\bigg( g\supset f\rightarrow \big(\exists h\in \mA (h\supset g) \wedge \exists h' \in \mB (h'\supset g)\big)\bigg)\Bigg).
\end{split}
\end{equation}
She maintains this condition until she sees an opportunity to finish off the game directly.\newline

Condition (\ref{e9}) says that at any stage of the game, we can always find a finite segment $f$ such that all strings in $\mA$ and $\mB$ agree with $f$. Moreover, for any finite extension $g$ of $f$, there are string $h$ and $h'$ in $\mA$ and $\mB$ respectively that agree with $g$. Therefore if we imagine we `cut off' the finite segment $f$, the remaining configuration in the game is almost identical to the starting position. We will specify later what counts as an opportunity for \bf{II} to win the game directly. 

Theorem \ref{t10} is established through a series of lemmas.

\begin{lemma}
Condition (\ref{e9}) is true in the starting position $(\P_1, \P_2, \alpha).$
\end{lemma}
\begin{proof}
Let $f$ be the empty function. Then condition (\ref{e9}) is just the density condition. 
\end{proof}

\begin{lemma}\label{l14}
If condition (\ref{e9}) is true in position $(\mA, \mB, \gamma)$, then no atomic formula or its negation separates $\mA$ and $\mB$.
\end{lemma}
\begin{proof}
There is an atomic formula or its negation separating $\mA$ and $\mB$ if and only if there exists $g\in 2^{<\omega}$ such that $|\textrm{dom}(g)|=1$, $g\subset h$ for all $h\in \mA$, and $g\not\subset h'$ for all $h'\in \mB$. We claim that condition (\ref{e9}) guarantees that this situation does not occur. There are the following three possibilities for $g$, let us consider them one by one. Firstly, if $g\subset f$, then by condition (\ref{e9}) all strings in $\mA$ and $\mB$ agree with $g$. Secondly, if $g$ disagrees with $f$, namely if dom($g$)$\subset$dom($f$) yet $g\not\subset f$, then by condition (\ref{e9}) again all strings in $\mA$ and $\mB$ disagree with $g$. Finally, suppose $\textrm{dom}(g)\cap \textrm{dom}(f)=\emptyset$. The finite segment $g\cup f$ is an extension of $f$. Therefore by condition (\ref{e9}) there are $h\in \mA$ and $h'\in \mB$ that agree with $g\cup f$ respectively. In particular they agree with $g$. 
\end{proof}

\begin{lemma}\label{l15}
If player \bf{I} makes a generalized finite splitting move, player \bf{II} can maintain condition (\ref{e9}) to the next round. 
\end{lemma}
\begin{proof}
Suppose the game is in position $(\mA, \mB, \gamma)$ and condition (\ref{e9}) holds. Suppose player \bf{I} makes a generalized finite left splitting move. He represents $\mA$ as a union $\mA_1\cup\ldots\cup \mA_n$ of $n$ pieces. He also picks ordinals $\gamma_1,\ldots, \gamma_n$ such that $\gamma_1\#\ldots\#\gamma_n\leq\gamma$. We claim that there must exist a finite segment $f'\supset f$ and a piece $\mA_k$ such that 
\begin{equation}\label{e10}
\forall g\in 2^{<\omega}\big( g\supset f'\rightarrow \exists h\in \mA_k (h\supset g)\big). 
\end{equation}
Suppose this is not the case for $\mA_1,\ldots, \mA_{n-1}$. Then we have
\begin{equation}\label{e11}
 \forall f'\supset f \exists g\in 2^{<\omega}\bigg(g\supset f' \wedge \forall h\in 2^\omega (h\supset g\rightarrow h\notin \mA_i) \bigg)
\end{equation}
for $i=1,\ldots, n-1$. Now let $f'=f$, apply condition (\ref{e11}) to $\mA_1$:
\[ \exists g_1\in 2^{<\omega}\bigg(g\supset f \wedge \forall h\in 2^\omega\big(  h\supset g_1\rightarrow h\notin \mA_1\big)\bigg).\]
This means that
\[ \exists g_1\in 2^{<\omega}\bigg(g\supset f \wedge \forall h\in 2^\omega \big( (h\supset g_1\wedge h\in \mA )\rightarrow h\in \mA_2\cup\ldots\cup \mA_n\big)\bigg).\]
Now let $f'=g_1$, apply condition (\ref{e11}) to $\mA_2$:
\[ \exists g_2\in 2^{<\omega}\bigg(g_2\supset g_1\wedge \forall h\in 2^\omega \bigg( (h\supset g_2 \wedge h\in \mA)\rightarrow h\in \mA_3\cup\ldots\cup \mA_n\big)\bigg).\]
Iterate this process, we get:
\[ \exists g_{n-1}\in 2^{<\omega}\bigg(g_{n-1}\supset g_{n-2} \forall h\in 2^\omega\bigg( (h\supset g_{n-1}\wedge h\in \mA)\rightarrow h\in \mA_n\big)\bigg).\]
This means that condition (\ref{e10}) is true for $\mA_n$ when we take $g_{n-1}$ for $f'$. Therefore the previous claim is true.

Now let player \bf{II} pick this piece $\mA_k$ and the corresponding ordinal $\gamma_k$. What remains to be done is a cosmetic surgery on $\mA_k$ and $\mB$. We throw away the strings that do not agree with $f'$. Let
\[ \mA'=\{h\in\mA_k|h\supset f'\}
\]
and let
\[\mB'=\{h'\in\mB|h'\supset f'\}.
\]
By Corollary \ref{c8}, if player \bf{II} has a winning strategy in this new position, she also has one in the old position. It is straightforward to check that condition (\ref{e9}) holds for $\mA'$ and $\mB'$. Let the game continue from the position $(\mA', \mB', \gamma_k)$. 

The case where \bf{I} plays a generalized finite right splitting move is similar.
\end{proof}

\begin{lemma}\label{l16}
Suppose condition (\ref{e9}) is true in position $(\mA, \mB, \gamma)$ with $\gamma<\omega^2$. If player \bf{I} makes a proper infinite splitting move, then player \bf{II} has a winning strategy in the rest of the game. 
\end{lemma}
\begin{proof}
Suppose the game is in position $(\mA, \mB, \gamma)$ and condition (\ref{e9}) is true. Suppose player \bf{I} makes a proper infinite splitting move. He represents $\mA$ as a union $\bigcup_{i\in \omega} \mA_i$, among which there are infinitely many non-empty pieces. He also chooses ordinals $\gamma_i, i\in \omega$ such that $\S\gamma_i\leq\gamma< \omega^2$. Note that there can only be finitely many $i\in\omega$ such that $\gamma_i$ is an infinite ordinal, for otherwise the natural sum $\S\gamma_i$ would be greater or equal to $\omega^2$, which is a contradiction. On the other hand there are infinitely many $i\in \omega$ such that $\mA_i$ is non-empty. Therefore there must exist $k\in \omega$ such that $\mA_k$ is non-empty and $\gamma_k$ is finite. 

If this situation occurs, we ask player \bf{II} to jump out of the strategy of maintaining condition (\ref{e9}) and go on to win the game directly. Let player \bf{II} pick this piece $\mA_k$ and $\gamma_k$. The game continues from the position $(\mA_k, \mB, \gamma_k)$. Apply Corollary \ref{c8} and suppose $\mA_k$ consists of a single string $h_A$. Note that since $\gamma_k$ is a finite number, in the rest of the game \bf{I} can only make finite right splitting moves. 

Suppose \bf{I} makes a finite right splitting move. He represents $\mB$ as a union $\mB_1\cup\ldots\cup\mB_n$ of $n$ pieces. Consider the witnessing finite segment $f$ for condition (\ref{e9}). Let $k=\sup(\mathrm{dom}(f))$. For any natural number $m$, let $h_m$ be the string in $\mB$ such that $h_m$ agrees with $h_A$ on the first $k+m$ elements of $\omega$. By condition (\ref{e9}) this string always exists. In other words, we have
\begin{equation}
\forall m\in \omega \ \exists h_m\in \mB \big(h_m\upharpoonright (k+m)=h_A\upharpoonright (k+m)\big).
\end{equation}
Among the $n$ pieces $\mB_1,\ldots, \mB_n$, there must be one piece $\mB_l$ that contains $h_m$ for infinitely many $m\in \omega$. Let player \bf{II} pick this piece $\mB_l$. She uses the same strategy if \bf{I} makes another finite right splitting move. She makes sure that at any stage of the remaining game, the piece $\mB_r$ at hand always contain $h_m$ for infinitely many $m\in \omega$. This guarantees that no atomic formula separates $\mA_k$ and $\mB$. Player \bf{II} wins when the game terminates after finitely many rounds. 
\end{proof}

Now we can see that the strategy prescribed above is indeed a winning strategy for player \bf{II}. Player \bf{II} hangs on in the game by maintaining condition (\ref{e9}). If her opponent plays a proper infinite splitting move at some point, she goes on to win the rest of the game by appealing to Lemma \ref{l16}. Otherwise she continues with condition (\ref{e9}). This is possible because of Lemma \ref{l15}. This strategy keeps her away from losing by Lemma \ref{l14}. Player \bf{II} will eventually prevail after finitely many rounds. Thus we have finished the proof of Theorem \ref{t10}.

The following are some examples of dense properties of infinite binary strings. By Corollary \ref{c11} none of these properties can be defined by a propositional $\LL$-formula of size less than $\omega^2$.

\begin{example}
\begin{enumerate}
\item $\P_1=\{f\in 2^\omega|f^{-1}(1)  \textrm{ is finite}\}$  `$f$ has finitely many ones'.

\item $\P_2=\{f\in 2^\omega||f^{-1}(1)|  \textrm{ is an odd number}\}$ `$f$ has an odd number of ones'.

\item $\P_3=\textrm{`$f$ codes a rational number'}$.

\item $\P_4=\textrm{`$f$ codes a decidable set'}$.
\end{enumerate}
\end{example}

This result is a first step in the application of EFB$_\alpha$. The game suggests many natural questions about the expressive power of $\LL$, for example:

\medskip
\noindent
\bf{Open Problems:\hspace{2mm}}
\begin{enumerate}

\item Can we use the game to prove higher complexity results for propositional $\LL$?

\item Can we find applications of the full game instead of the propositional fragment? 

\item Can we prove complexity results for other nice complexity measures? 

\end{enumerate}
These remain questions for further study.

Finally, we point out that new problems arise when we move from $\LL$ to $L_{\omega_2\omega}$ and even higher infinitary languages. Suppose we want to define the notion of size for $L_{\omega_2\omega}$-formulas in such a way that it is coherent with our definition of size for $\LL$-formulas. To illustrate the problem, consider the uncountable conjunction $\Phi=\bigwedge_{i\in\omega_1}\phi_i$, where $\phi_i$ is an atomic formula for all $i\in\omega_1$. What is the right size of $\Phi$? There are the following three desiderata:
\begin{enumerate}
\item If $\theta$ is an $\LL$-formula, then the size of $\theta$ is the same as it is given in Definition \ref{d8}.
\item The size of $\Phi$ is the supremum of the size of its proper subformulas, namely $\mathrm{s}(\Phi)=\sup_{\alpha<\omega_1}\mathrm{s}(\bigwedge_{i\in\alpha}\phi_i)$.
\item The size of $\Phi$ is at least $\omega_1$.
\end{enumerate}
It turns out that these three conditions cannot be satisfied simultaneously. Suppose there is a complexity measure $\mathrm{s}$ satisfying conditions 1 and 2.  Consider any infinite ordinal $\alpha<\omega_1$. By condition 1, we have that $\mathrm{s}(\bigwedge_{i\in\alpha}\phi_i)=\mathrm{s}(\bigwedge_{i\in\omega}\phi_{p(i)})=\omega$, where $p$ is a bijection from $\omega$ to $\alpha$. This together with condition 2 renders the size of $\Phi$ to be $\omega$, contrary to condition 3.

\bibliography{EF.bib}
\printindex

\end{document}